\documentclass[reqno,12pt]{amsart}
\usepackage{txfonts}
 \usepackage{amssymb,latexsym}
\usepackage{amsmath,amsfonts,amsthm}
\usepackage{graphicx}
\usepackage{indentfirst}
\usepackage{fancyhdr}
\usepackage{geometry}
\fancypagestyle{plain}{%
 \fancyhead[LO,RE]{\usebox{\mygraphic}}
 
 }
\begin{document}
\setlength{\abovedisplayshortskip}{2mm}
\setlength{\belowdisplayshortskip}{2mm}
\setlength{\abovedisplayskip}{2mm}
\setlength{\belowdisplayskip}{2mm}

\newtheorem{theorem}{Theorem}[section]
\newtheorem{corollary}[theorem]{Corollary}  
\newtheorem{lemma}[theorem]{Lemma}
\newtheorem{definition}[theorem]{Definition}
\newtheorem{proposition}[theorem]{Proposition}
\newtheorem{remark}[theorem]{Remark}
\newtheorem{example}[theorem]{Example}
\newtheorem{theoremalph}{Theorem}
 \renewcommand\thetheoremalph{\Alph{theoremalph}}

\newcommand{\bta}{\begin{theoremalph}}
\newcommand{\eeta}{\end{theoremalph}}
\newcommand{\bth}{\begin{theorem}}
\newcommand{\eeth}{\end{theorem}}
\newcommand{\ble}{\begin{lemma}}
\newcommand{\ele}{\end{lemma}}
\newcommand{\bco}{\begin{corollary}}
\newcommand{\eco}{\end{corollary}}
\newcommand{\bde}{\begin{definition}}
\newcommand{\ede}{\end{definition}}
\newcommand{\bpr}{\begin{proposition}}
\newcommand{\epr}{\end{proposition}}
\newcommand{\bre}{\begin{remark}}
\newcommand{\ere}{\end{remark}}
\newcommand{\beg}{\begin{example}}
\newcommand{\eeg}{\end{example}}

\newcommand{\beq}{\begin{equation}}
\newcommand{\eeq}{\end{equation}}
\newcommand{\ben}{\begin{equation*}}
\newcommand{\een}{\end{equation*}}
\newcommand{\beqn}{\begin{eqnarray}}
\newcommand{\eeqn}{\end{eqnarray}}
\newcommand{\be}{\begin{eqnarray*}}
\newcommand{\ee}{\end{eqnarray*}}
\newcommand{\ban}{\begin{align*}}
\newcommand{\ean}{\end{align*}}
\newcommand{\bal}{\begin{align}}
\newcommand{\eal}{\end{align}}
\newcommand{\bad}{\aligned}
\newcommand{\ead}{\endaligned}
\newcommand{\lan}{\langle}
\newcommand{\ran}{\rangle}

\newcommand{\na}{\nabla}
\newcommand{\vp}{\varphi}
\newcommand{\La}{\Lambda}
\newcommand{\la}{\lambda}
\newcommand{\Om}{\Omega}
\newcommand{\ta}{\theta}
\newcommand{\fr}{\frac}
\newcommand{\iy}{\infty}
\newcommand{\ve}{\varepsilon}
\newcommand{\pa}{\partial}
\newcommand{\al}{\alpha}
\newcommand{\mr}{\mathbb{R}^n}
\newcommand{\bu}{\bullet}
\newcommand{\si}{\sigma}
\newcommand{\Vol}{{\rm Vol}}
\newcommand{\Div}{{\rm div}}
\newenvironment{sequation}{\begin{equation}\small}{\end{equation}}
\newenvironment{tequation}{\begin{equation}\tiny}{\end{equation}}

\title
{$\mathcal{W}$-Entropy formulae  and  differential Harnack estimates for  porous medium equations on Riemannian manifolds}
\author{Yu-Zhao Wang}
\address{School of Mathematical Sciences, Shanxi University, Taiyuan, 030006, Shanxi, China}
\email{wangyuzhao@sxu.edu.cn}
\thanks{The author is supported by the National Science Foundation of China(NSFC, 11701347).}

\maketitle
\maketitle \numberwithin{equation}{section}
\maketitle \numberwithin{theorem}{section}
\setcounter{tocdepth}{2}
\setcounter{secnumdepth}{2}

\begin{abstract}
In this paper, we prove Perelman type $\mathcal{W}$-entropy formulae and global differential Harnack estimates  for positive solutions to porous medium equation on  closed Riemannian manifolds with Ricci curvature bounded below. As applications, we derive Harnack inequalities and Laplacian estimates.

\vspace{2mm}
\textbf{Mathematics Subject Classification (2010)}. Primary 58J35,	35K92; Secondary 35B40,35K55

\textbf{Keywords}. Porous medium equation, Perelman type entropy formula, differential Harnack  estimates, Bakry-\'Emery Ricci curvature.
\end{abstract}

\section{Introduction and main results}

Monotonicity formula and differential Harnack inequality are two  important tools in geometric analysis. The more spectacular one is the entropy monotonicity formula discovered by G.Perelman \cite{P} and related differential Harnack inequality for the conjugate heat equation under the Ricci flow. More precisely, let $(M,g(t))$ be the closed $n$-dimensional Riemannian manifolds along the Ricci flow and $(g(t), f(t), \tau(t))$ be a solution to conjugate heat equation coupled with Ricci flow
\begin{equation}\label{adjoint}
\partial_tg=-2{\rm Ric},\quad\partial_t f=-\Delta f+|\nabla f|^2+R+\frac n{2\tau},\quad\partial_t\tau=-1.
\end{equation}
Perelman \cite{P} introduced the following $\mathcal{W}$-entropy
$$
\mathcal{W}(g,f,\tau):=\int_M\Big(\tau(R+|\nabla f|^2)+f-n\Big)\frac{e^{-f}}{(4\pi\tau)^{\frac n2}}\,dV
$$
and proved its monotonicity
\begin{equation}\label{Pentropy}
\frac{d}{dt}\mathcal{W}(g,f,\tau)=2\tau\int_M\Big|R_{ij}+\nabla_i\nabla_jf-\frac1{2\tau}g_{ij}\Big|^2\frac{e^{-f}}{(4\pi\tau)^{\frac n2}}\,dV\ge0,
\end{equation}
where $\tau>0$ and $\int_M{(4\pi\tau)^{-\frac n2}}e^{-f}dV=1$. The static point of $\mathcal{W}$-entropy is called gradient shrinking Ricci soliton
\begin{equation}\label{soliton}
R_{ij}+\nabla_i\nabla_jf=\frac1{2\tau}g_{ij},
\end{equation}
which can be viewed as a model with singularity when studying the singularity formation of solutions of the Ricci flow.  Moreover, when $H(x,t)={(4\pi\tau)^{-\frac n2}}e^{-f}$ is a fundamental solution to the conjugate heat equation \eqref{adjoint}, Perelman's differential Harnack inequality holds \cite{P,NiLYH}
\begin{equation}\label{PLYH}
v_H:=\Big(\tau(R+2\Delta f-|\nabla f|^2)+f-n\Big)H\le0.
\end{equation}

A feature of Perelman's entropy monotonicity formula is that it holds at any dimension and without assumption of curvature condition. After Perelman's work, an interesting direction is that of finding entropy monotonicity formulae for other geometric evolution equations.  There were some development in this direction, L. Ni \cite{Nientropy}  derived the entropy monotonicity formula for the linear heat equation on Riemannian manifolds with nonnegative Ricci curvature,
\begin{equation}\label{Nientropy}
\frac{d}{dt}\mathcal{W}(f,\tau)=-2\tau\int_M\left(\Big|\nabla_i\nabla_jf-\frac1{2\tau}g_{ij}\Big|^2+R_{ij}f_if_j\right)u\,dV,
\end{equation}
where $u=(4\pi\tau)^{-\frac n2}e^{-f}$ is a positive solution to the heat equation $\partial_tu=\Delta u$ with $\int_Mu\,dV=1$, $\frac{d\tau}{dt}=1$ and $\mathcal{W}(f,\tau)$ is defined by
$$
\mathcal{W}(f,\tau):=\int_M\Big(\tau|\nabla f|^2+f-n\Big)u\,dV.
$$

When the Ricci curvature of $M$ is bounded below, Li-Xu \cite{LX} established some Perelman-Ni type entropy formulae for the linear heat equation and proved their monotonicity.

A natural question is how to establish the entropy formula for the nonlinear equation
on Riemannian manifolds.  Kotschwar-Ni \cite{KoNi} and Lu-Ni-Vazquez-Villani \cite{LNVV} obtained the entropy monotonicity formula for the $p$-Laplaican heat equation and the porous medium equation on compact Riemannian manifolds with nonnegative Ricci curvature respectively.
In \cite{WC2}, the author proved the entropy monotonicity formula for positive solution to the doubly nonlinear diffusion equation on the closed Riemannian manifolds with nonnegative Ricci curvature.

The first step in the study of the $\mathcal{W}$-entropy on Riemannian
manifolds with negative lower bound of Ricci curvature
is to find the suitable quantity to define the $\mathcal{W}$-entropy.
In \cite{LiLi2, LiLi3}, S. Li and X.-D. Li gave a new idea to introduce  the correct quantity for the definition of the $\mathcal{W}$-entropy for the heat equation of the Witten Laplacian on Riemannian manifolds with lower bound of the  infinity dimensional Bakry-Emery Ricci curvature. Motivated by their works, we can obtain the Perelman type $\mathcal{W}$-entropy monotonicity formula for the porous medium equation on closed Riemannian manifolds with Ricci curvature (or Bakry-Emery Ricci curvature) bounded below.

\begin{theorem}\label{KPMEentropy}
Let $(M,g)$ be a closed Riemannian manifold with  Ricci curvature bounded below by $-K(K\ge0)$. Suppose~$u$~be a smooth positive solution to the porous medium equation
\begin{equation}\label{PME}
\partial_tu=\Delta u^{\gamma}
\end{equation}
 and $v=\frac{\gamma}{\gamma-1}u^{\gamma-1}$ the pressure function. For any $\gamma>1$,
define the   Perelman-type $\mathcal{W}$-entropy
\begin{equation}\label{WKentropy}
\mathcal{W}_K(v,t):=\sigma_K\beta_K\int_M\left[\gamma\frac{|\nabla v|^2}{v}-\left(\frac{1}{\beta_K}+\frac{\dot{\sigma}_K}{\sigma_K}\right)\right]vu\,dV,
\end{equation}
Then we have
\begin{align}\label{WKDentropy}
\frac{d}{dt}\mathcal{W}_K(v,t)
\le&-2\sigma_K\beta_K\int_M (\gamma-1) \left|\nabla_i\nabla_jv+\frac{\eta_K}{n(\gamma-1)}g_{ij}\right|^2vu\,dV\\
&-2\sigma_K\beta_K\int_M\left[(\gamma-1)({\rm Ric}+Kg)(\nabla v,\nabla v )+((\gamma-1)\Delta v+\eta_K)^2\right]vu\,dV,\notag
\end{align}
where $a=\frac{n(\gamma-1)}{n(\gamma-1)+2}$, $\kappa=K\sup\limits_{M\times(0,T]}u^{\gamma-1}$, $\sigma_K=\left(\frac{e^{2{\kappa}t}-1}{2\kappa}\right)^{a}$, $\beta_K=\frac{\sinh(2\kappa t)}{2\kappa}$ and  $\eta_K=\frac{2a\kappa}{1-e^{-2\kappa t}}$.
Moreover, if ${\rm Ric}\ge-Kg$ for $K\ge0$, then  $\mathcal{W}_{K}(v,t)$ is monotone decreasing along the porous medium equation \eqref{PME}.
\end{theorem}
\begin{remark}
When $K=0$, $\sigma_0=t^{a}$, $\beta_0=t$ and $\eta_0=\frac{a}t$, the entropy formula \eqref{WKDentropy} in Theorem \ref{KPMEentropy} reduced the result of  Lu-Ni-Vazquez-Villani in \cite{LNVV},
\begin{align}\label{PMEentropy2}
\frac{d}{dt}\mathcal{W}_0(v,t)=&-2(\gamma-1)t^{a+1}\int_M\left[\Big|\nabla_i\nabla_jv
+\frac{a}{n(\gamma-1)t}g_{ij}\Big|^2+ {\rm Ric}(\nabla v,\nabla v)\right]vu\,dV\notag\\
&-2t^{a+1}\int_M\left[(\gamma-1)\Delta v+\frac {a}t\right]^2 vu\,dV,
\end{align}
where
\begin{equation}\label{PMEentropy1}
\mathcal{W}_0(v,t)= t^{a+1}\int_M\left(\gamma\frac{|\nabla v|^2}{v}-\frac{a+1}t\right)vu\,dV.
\end{equation}
\end{remark}

A weighted Riemannian manifold is a Riemannian manifold $(M,g)$ with a smooth measure $d\mu:= e^{-f}\,{dV}$, denoted by $(M,g, d\mu)$, where $f$ is a smooth function on $M$. The weighted Riemannian manifold carries a natural analog of the Ricci curvature, that is, the $m$-Bakry-\'Emery Ricci curvature, which is defined as
$$
{\rm Ric}_f^m:={\rm Ric}+\nabla\nabla f-\frac{\nabla f\otimes \nabla f}{m-n},\quad (n\leq m\leq \infty).
$$
In particular, when $m=\infty$, ${\rm Ric}^{\infty}_f={\rm Ric}_f:={\rm Ric}+\nabla\nabla f$ is the classical Bakry--\'Emery Ricci curvature, which was introduced in the study of diffusion processes and functional inequalities including the Poincar\'e and the logarithmic Sobolev inequalities(see \cite{BGL} for a comprehensive introduction),  then it is extensively investigated in the theory of the Ricci flow (for example, the gradient shrinking Ricci soliton equation \eqref{soliton} is precisely ${\rm Ric}_f =\frac{1}{2\tau} g$);  when $m=n$ if and only if $f$ is a constant function. There is also a natural analog of the Laplacian, namely, the so-called weighted Laplacian, denoted by $\Delta_f=\Delta-\nabla f\cdot\nabla$, which is a self-adjoint operator in $L^2(M, d\mu)$.

A nature question is which contents can be generalized to the weighted manifolds and what are the advantages and applications for the weighted case. There are a lot of progress in this direction, for instance, gradient estimates and Liouville theorems for symmetric diffusion operators $\Delta_f$ \cite{LiXD1}, some comparison geometry for the Bakry-Emery Ricci tensor \cite{WW} etc..  In view of the innovation and importance of above-mentioned Perelman $\mathcal{W}$-entropy formulae,  one tried to get the entropy formula for the weighted case, the first work is established by X.-D. Li for the weighted heat equation $\partial_t u=\Delta_fu$ in \cite{LiXD2,LiXD3},
  \begin{align}\label{Lientropy}
\frac{d}{dt}\mathcal{W}_f(v,\tau)=&-2t\int_M\left(\Big|\nabla_i\nabla_jv-\frac1{2t}g_{ij}\Big|^2+{\rm Ric}^m_f(\nabla v,\nabla v)\right)u\,d\mu\notag\\
&-\frac{2t}{m-n}\int_M\Big(\nabla f\cdot\nabla v+\frac{m-n}{2t}\Big)^2u\,d\mu,
\end{align}
where the $\mathcal{W}$-entorpy is defined by
$$
\mathcal{W}_f(v,t):=\int_M\Big(t|\nabla v|^2+v-m\Big)u\,d\mu,\quad u = \frac{e^{-v}}{(4\pi t)^{m/2}}.
$$
In particular, if the $m$-Bakry-\'Emery Ricci curvature is nonnegative, then $\mathcal{W}_f(v,\tau)$ is monotone decreasing along the weighed heat equation. When $m=n$, $f=const.$, \eqref{Lientropy} reduces to \eqref{Nientropy}.

In \cite{LiLi}, when $n\le m\in \mathbb{N}$,  S. Li and X.-D. Li gave a new proof of the $\mathcal{W}$-entropy formula \eqref{Lientropy} by using of the warped product approach and a natural geometric interpretation for the third term in \eqref{Lientropy}. Moveover, they extended the $\mathcal{W}$-entropy formula to the weighted heat equation on the weighted compact Riemannian manifolds with time dependent metrics and potentials under satisfying the curvature dimensional condition $CD(K,m)$ for some negative constant $K$. In \cite{LiLi2}, S. Li and X.-D. Li obtained the $\mathcal{W}$-entropy formula  on compact Riemannian manifolds with $(K,m)$-super Perelman Ricci flow, where $K\in \mathbb{R}$ and $m\in [n,\infty]$ are two constants. In their paper \cite{LiLi4}, S. Li and X.-D. Li pointed out that there is an essential and deep connection between the definition of the $W$-entropy and the Hamilton type Harnack inequality on complete Riemnnian manifolds with the $CD(K,m)$-condition. Recently, they \cite{LiLi3} introduced Perelman's $\mathcal{W}$-entropy  along geodesic flow on the Wasserstein space over Riemannian manifolds and proved a rigidity theorem. For further related study, see \cite{LiLi3, LiLi5,LiLi6}.

For the nonlinear case,  combining the analogous methods in \cite{KoNi}, \cite{LNVV} and \cite{LiXD2}, Wang-Yang-Chen \cite{WYC} and Huang-Li \cite{HL} obtained the entropy monotonicity formulae for the weighted $p$-Laplacian heat equation and the weighted porous medium equation with nonnegative $m$-Bakry-\'Emery Ricci curvature respectively. In \cite{WYZ}, the author expanded the entropy formulae of Kotschwar-Ni \cite{KoNi} and Wang-Yang-Chen \cite{WYC} to the case  where the $m$-dimensional
Bakry-\'Emery Ricci curvature is bounded from below.

 Inspired by above works, we can obtain entropy monotonicity formula for positive solution to the weighted porous medium equation with $m$-Bakry-\'Emery Ricci curvature bounded below,  which is a natural generalization of  Theorem \ref{KPMEentropy}.

\begin{theorem}\label{WKPMEentropy}
Let~$(M,g,d\mu)$ be a closed weighted Riemannian manifold with  $m$-Barky-Emery-Ricci curvature bounded below,  i.e. ${\rm Ric}_f^m\ge-Kg$, $K\ge0$. Suppose~$u$~be a smooth positive solution to equation \eqref{WPME}
\begin{equation}\label{WPME}
\partial_tu=\Delta_f u^{\gamma}
\end{equation}
 and $v=\frac{\gamma}{\gamma-1}u^{\gamma-1}$ the pressure function. For any $\gamma>1$,
the  Perelman-type $\mathcal{W}$-entropy is defined by
\begin{equation}\label{WKentropy}
\mathcal{W}_K(v,t):=\bar{\sigma}_K\bar{\beta}_K\int_M\left[\gamma\frac{|\nabla v|^2}{v}-\left(\frac{1}{\bar{\beta}_K}+\frac{\dot{\bar{\sigma}}_K}{\bar{\sigma}_K}\right)\right]vu\,d\mu,
\end{equation}
Then we have
\begin{small}
\begin{align}\label{WKPMEent}
&\frac{d}{dt}\mathcal{W}_K(v,t)
\le-2\bar{\sigma}_K\bar{\beta}_K\int_M (\gamma-1) \left[\left|\nabla_i\nabla_jv+\frac{\bar{\eta}_K}{n(\gamma-1)}g_{ij}\right|^2+({\rm Ric}^m_f+Kg)(\nabla v,\nabla v )\right]vu\,d\mu\notag\\
&-2\bar{\sigma}_K\bar{\beta}_K\int_M\left[\Big((\gamma-1)\Delta_f v+\bar{\eta}_K\Big)^2+\frac{\gamma-1}{m-n}\left(\langle\nabla v,\nabla f\rangle-(m-n)\frac{\bar{\eta}_K}{m(\gamma-1)}\right)^2\right]vu\,d\mu.
\end{align}
\end{small}
where $\bar{a}=\frac{m(\gamma-1)}{m(\gamma-1)+2}$, ${\kappa}=K\sup\limits_{M\times(0,T]}u^{\gamma-1}$, $\bar{\sigma}_K=\left(\frac{e^{2{\kappa}t}-1}{2\kappa}\right)^{\bar{a}}$, $\bar{\beta}_K=\frac{\sinh(2{\kappa}t)}{2{\kappa}}$, $\bar{\eta}_K=\frac{2\bar{a}{\kappa}}{1-e^{-2{\kappa}t}}$.
Moreover, if ${\rm Ric}_f^m\ge-Kg$ for $K\ge0$, then  entropy $\mathcal{W}_{K}(v,t)$ is monotone decreasing along the weighted porous medium equation \eqref{WPME}.
\end{theorem}

In the second part of this paper, we study the differential Harnack inequality for the porous medium equation. Such inequality was first proved by  Li and Yau \cite{LY} for solutions to the heat equation on Riemannian manifolds, i.e. if $u$ is a positive solution to $\partial_tu=\Delta u$ on $M$ with nonnegative Ricci curvature, then
\begin{equation}\label{LY}
\frac{|\nabla u|^2}{u^2}-\frac{u_t}{u}\le\frac{n}{2t}.
\end{equation}
Later on, it has been extensively investigated in other geometric evolution equations, such as Hamilton's estimates for the Ricci flow and the mean curvature flow, the corresponding results for the K\"ahler Ricci flow and the Gauss curvature flow were proved by H. Cao and B. Chow, Perelman's differential Harnack inequality \eqref{PLYH} for the conjugate heat equaiton under Ricci flow, etc.. For more progress in this direction, see the survey \cite{Ni} and references therein.

It is a long time question of proving the sharp Li-Yau Harnack inequality for positive solution of the heat equation on Riemannian manifolds with negative Ricci curvature bound (see P.393 in \cite{CLN}). There are many works in this direction. Let ${\rm Ric}\ge-K$ for some $K>0$, the original result of Li-Yau \cite{LY} is,
$$
\frac{|\nabla u|^2}{u^2}-\alpha\frac{u_t}{u}\le\frac{\alpha^2}{2(\alpha-1)}nK+\alpha^2\frac{n}{2t},
$$
where $\alpha$ is a constant and  $\alpha>1$.
In \cite{Ham93}, Hamilton  proved
$$
\frac{|\nabla u|^2}{u^2}-e^{2Kt}\frac{u_t}{u}\le e^{4Kt}\frac{n}{2t}.
$$
Recently, Li-Xu \cite{LX} obtained some new Li-Yau type estimates,
\begin{align}\label{LiXu}
\frac{|\nabla u|^2}{u^2}-\left(1+\frac{2}3Kt\right)\frac{u_t}{u}\le&\frac{n}{2t}+\frac{nK}{2}\Big(1+\frac13Kt\Big),\\
\frac{|\nabla u|^2}{u^2}-\left(1+\frac{\sinh (Kt)\cosh (Kt)-Kt}{\sinh^2( Kt)}\right)\frac{u_t}{u}\le&\frac{nK}{2}\Big(1+\coth(Kt)\Big).\notag
\end{align}
In \cite{LiLi,LiLi3}, S. Li and X.-D. Li proved an analog of the Harnack inequality \eqref{LiXu} for the weighted heat equation on weighted complete Riemannian manifolds with  ${\rm Ric}^m_f\ge-K$ . B.Qian \cite{QianB} gave a further generalization under the  proper assumptions of $\alpha(t)$ and $\varphi(t)$,
\begin{equation}\label{QLY}
\frac{|\nabla u|^2}{u^2}-\alpha(t)\frac{u_t}{u}\le\varphi(t).
\end{equation}
There are some results on Li-Yau, Hamilton and Li-Xu type differential Harnack estimates for the porous medium equation \cite{LNVV, HHL, WC1}, the $p$-heat equation \cite{KoNi,WYC,WYZ} and the doubly nonlinear diffusion equation \cite{WC2} on Riemannian manifolds with Ricci curvature bounded below.

In the second part of this paper, we obtain  Qian type differential Harnack estimate for the porous medium  equation on the closed Riemannian manifold with Ricci curvature bounded below, which can also be generalized to the case of the weighted Riemanian manifolds with the $m$-Bakry-\,Emery Ricci curvature bound below.

Now we first give two assumptions, let $\sigma(t)\in C^1(M)$ and satisfy (See \cite{QianB}).
\begin{description}
  \item[(A1)] For a $t>0$, $\sigma(t)>0$, $\sigma'(t)>0$, $\lim\limits_{t\to0}\sigma(t)=0$ and $\lim\limits_{t\to0}\frac{\sigma(t)}{\sigma'(t)}=0$;
  \item[(A2)] For any $T>0$, $\frac{(\sigma')^2}{\sigma}$ is continuous and integrable on the interval $[0,T)$.
\end{description}

\begin{theorem}\label{pmeGK}
Let $(M^n,g)$ be a closed Riemannian manifold with ${\rm Ric}\ge-Kg$ for $K\ge0$. Suppose  $u(x,t)$ be a positive solution to equation \eqref{PME} and $v$ the  pressure function. For any $\gamma>1$, we have
\begin{equation}\label{PMELYH1}
\frac{|\nabla v|^2}{v}-\alpha(t)\frac{v_t}{v}\le\varphi(t).
\end{equation}
Here
\begin{align}\label{PMEalphavarphi}
\alpha(t)=1+\frac{2{\kappa}}{\sigma}\int^t_0\sigma(s)ds,\quad
\varphi(t)={\kappa}a
+\frac{{\kappa}^2a}{\sigma}\int_0^t\sigma(s)ds+\frac{a}{4\sigma}\int^t_0\frac{(\sigma'(s))^2}{\sigma(s)}ds,
\end{align}
and
$\sigma(t)$ satisfies the assumptions $\mathbf{(A1)}$ and $\mathbf{(A2)}$,  $a=\frac{n(\gamma-1)}{n(\gamma-1)+2}$ and  ${\kappa}=K\sup\limits_{M\times(0,T]}u^{\gamma-1}$.
\end{theorem}

\begin{remark}
\begin{enumerate}
      \item When $K=0$, $\alpha(t)=1$, $\sigma(t)=t^2$ and $\varphi(t)=\frac at$, the estimate \eqref{PMELYH1} in Theorem \ref{pmeGK} reduces to the Aronson-Benilan's estimate in \cite{LNVV}, i.e.
          $$
          \frac{|\nabla v|^2}{v}-\frac{v_t}{v}\le\frac at.
          $$
      \item When $K>0$, $\sigma(t)=t^2$,  $\alpha(t)=1+\frac{2\kappa t}3$ and $\varphi(t)=\frac at+a\kappa\left(1+\frac{\kappa t}3\right)$, the estimate \eqref{PMELYH1} in Theorem \ref{pmeGK} reduces to the Li-Xu type estimate in \cite{HHL}
          $$
          \frac{|\nabla v|^2}{v}-\left(1+\frac{2\kappa t}3\right)\frac{v_t}{v}\le\frac at+a\kappa\left(1+\frac{\kappa t}3\right).
          $$
      \item When $K>0$, $\sigma(t)=\sinh^2(\kappa t)$,  $\alpha(t)=1+\frac{\sinh(\kappa t)\cosh(\kappa t)-\kappa t}{\sinh^2(\kappa t)}$ and $\varphi(t)=a\kappa(1+\coth(\kappa t))$, the estimate \eqref{PMELYH1} in Theorem \ref{pmeGK} reduces to another Li-Xu type estimate in \cite{HHL}
          $$
          \frac{|\nabla v|^2}{v}-\left(1+\frac{\sinh(\kappa t)\cosh(\kappa t)-\kappa t}{\sinh^2(\kappa t)}\right)\frac{v_t}{v}\le a\kappa(1+\coth(\kappa t)).
          $$
         \item All of above gradient estimates are valid for the weighted case by similar method.
    \end{enumerate}
\end{remark}

There are some applications for differential Harnack estimates, including Harnack inequalities and Laplacian estimates.
Integrating the estimate \eqref{PMELYH1} in Theorem \ref{pmeGK} along a
minimizing path between two points, we can get the Harnack inequalities for positive solutions to the porous medium equation \eqref{PME}.

\begin{corollary}\label{Harnack}
For any $(x_1,t_1)$ and $(x_2,t_2)$ with $0<t_1\le t_2<T$, we have
$$
v(x_1,t_1)-v(x_2,t_2)\le v_{max}\int^{t_2}_{t_1}\frac{\varphi(t)}{\alpha(t)}dt
+\frac{1}{4}\frac{d(x_2,x_1)^{2}}{(t_2-t_1)^{2}}\int^{t_2}_{t_1}\alpha(t)dt
$$
and
$$
\frac{v(x_1,t_1)}{v(x_2,t_2)}\le\exp\left(\int^{t_2}_{t_1}\frac{\varphi(t)}{\alpha(t)}dt
+\frac{1}{4}\frac1{v_{max}}\frac{d(x_2,x_1)^{2}}{(t_2-t_1)^{2}}\int^{t_2}_{t_1}\alpha^{\frac1{p-1}}(t)dt\right),
$$
where $v_{max}=\sup_{M\times[0,T)}v$.
\end{corollary}

\begin{corollary}\label{pfLaEst}
Under the same assumptions as in Theorem \ref{pmeGK}, define $\beta(t)$ by $\frac{\alpha(t)-1}{\alpha(t)}=(\gamma-1)(\beta(t)-1)$ for $\alpha(t)>1$ and such that $1<\beta(t)<\frac{\gamma}{\gamma-1}$. Assume that ${\rm Ric}\ge-K$ for some $K\ge0$, then
\begin{equation}\label{pfLaE}
\Delta(v^{\beta})\ge-\frac{\beta}{\alpha(\gamma-1) } v_{max}^{\beta-1}\varphi(t),
\end{equation}
where $v_{max}:=\sup_{M\times[0,T]}v$.
\end{corollary}

This paper is organized as follows. In Section 2, we establish some evolution equations and then prove the entropy monotonicity formulae,  i.e. Theorem \ref{KPMEentropy} and Theorem \ref{WKPMEentropy}. In section 3, we obtain the Qian type differential Harnack estimate. Finally, Harnack inequalities and Laplacian estimates are derived as applications.

\section{Entropy monotonicity formulae}

Let~$(M,g)$ be a closed Riemannian manifold. Suppose~$u$~be a smooth solution to \eqref{PME} and $v=\frac{\gamma}{\gamma-1}u^{\gamma-1}$ be the pressure function, define the parabolic operator
$$
\square:=\partial_t-(\gamma-1)v\Delta.
$$
Thus $v$ satisfies the equation
$$
\label{pressure}\square v =|\nabla v|^2.
$$

\begin{lemma}\label{pmeBochner} Let  $\alpha,\beta$ be two constants and $w=|\nabla v|^2$, $F_{\alpha}=\alpha\frac{v_t}{v}-\frac{|\nabla v|^2}{v}$,
then we have the following evolution equations(See \cite{LNVV}),
\begin{align}
\label{pmeBochner0}\square v_t=&(\gamma-1) v_t\Delta v+2\langle\nabla v,\nabla v_t\rangle,\\
\label{pmeBochner2}\square v^{\beta}=&\beta\big(\beta+\gamma-\beta\gamma\big)v^{\beta-1}w,\\
\label{pmeBochner3}\square w=&2\langle\nabla v,\nabla w\rangle
+2(\gamma-1)w\Delta v
-2(\gamma-1)v\Big(|\nabla\nabla v|^2+{\rm Ric}(\nabla v,\nabla v)\Big)\\
\label{pmeBochner4}\square F_{\alpha}=&2\gamma \left\langle \nabla v ,\nabla F_{\alpha}\right\rangle+2(\gamma-1)\Big(|\nabla\nabla v|^2+{\rm Ric}(\nabla v,\nabla v)\Big)
+\Big(F_1^2+(\alpha-1)\Big(\frac{v_t}v\Big)^2\Big).
\end{align}
\end{lemma}

By means of the  Bochner-type formulae in Lemma \ref{pmeBochner}, we have the following integral formulae, which  are useful for the proof of the $W$-entropy
formulae (See\cite{LNVV}).
\begin{lemma}\label{pmeint}
Let $u$ be a positive solution to \eqref{PME}, we have
\begin{align}\label{pmeint1}
\frac{d}{dt}\int_Mvu\,dV=&(\gamma-1)\int_M(\Delta v)vu\,d\mu=-\gamma\int_M|\nabla v|^2u\,dV.\\
\label{pmeint2}
\frac{d^2}{dt^2}\int_Mvu\,dV=&2(\gamma-1)\int_M\left(|\nabla\nabla v|^2+{\rm Ric}(\nabla v,\nabla v)+(\gamma-1)(\Delta v)^2\right)vu\,dV.
\end{align}
\end{lemma}

Applying integral formulae in Lemma \ref{pmeint}, we can deduce the $\mathcal{W}$-entropy formula for the porous medium equation on the closed Riemannian manifolds with  Ricci curvature lower bound.

\begin{proof}[\bf{Proof of Theorem \ref{KPMEentropy}}] Firstly, Boltzmann-Nash type entropy is defined by
$$
\mathcal{N}_K(t):=-\sigma_K(t)\int_Mvu\, dV,
$$
where $\sigma_K(t)$ is a function of $t$, then \eqref{pmeint1}  implies that
\begin{align}\label{Kpment1}
\frac{d}{dt}\mathcal{N}_K(t)=&-\dot{\sigma}_K\int_Mv u\,d\mu-\sigma_K\frac{d}{dt}\int_Mvu\,dV\notag\\
=&-\sigma_K\int_M\left((\gamma-1)\Delta v+(\log\sigma_K)'\right)vu\,dV
\end{align}

By \eqref{pmeint2}, \eqref{Kpment1} and the assumption ${\kappa}=K\sup_{M\times[0,T)}u^{\gamma-1}=\frac{\gamma-1}{\gamma}K\sup_{M\times[0,T)}v$, one has
\begin{small}
\begin{align}\label{Kpment3}
\frac{d^2}{dt^2}\mathcal{N}_K(t)
=&-\sigma_K\frac{d^2}{dt^2}\int_Mvu\,dV
-2\dot{\sigma}_K\frac{d}{dt}\int_Mvu\,dV-\ddot{\sigma}_K\int_Mv u\,dV\notag\\
=&-2(\gamma-1)\sigma_K\int_M\left[|\nabla\nabla v|^2+{\rm Ric}(\nabla v,\nabla v)+(\gamma-1)(\Delta v)^2\right]vu\,dV\notag\\
&+2\frac{\dot{\sigma}_K}{\sigma_K}\frac{d}{dt}\mathcal{N}_K
+\left(\frac{\ddot{\sigma}_K}{\sigma_K}
-\frac{2\dot{\sigma}_K^2}{\sigma_K^2}\right)\mathcal{N}_K\notag\\
=&-2(\gamma-1)\sigma_K\int_M\left[|\nabla\nabla v|^2+({\rm Ric}+Kg)(\nabla v,\nabla v)+(\gamma-1)(\Delta v)^2\right]vu\,dV\notag\\
&+2(\gamma-1)\sigma_K\int_MK|\nabla v|^2vu\,dV+2\frac{\dot{\sigma}_K}{\sigma_K}\frac{d}{dt}\mathcal{N}_K
+\left(\frac{\ddot{\sigma}_K}{\sigma_K}
-\frac{2\dot{\sigma}_K^2}{\sigma_K^2}\right)\mathcal{N}_K\notag\\
\le&-2(\gamma-1)\sigma_K\int_M \left[|\nabla\nabla v|^2+({\rm Ric}+Kg)(\nabla v,\nabla v )+(\gamma-1)(\Delta v)^2\right]vu\,dV\notag\\
&+2\left(\frac{\dot{\sigma}_K}{\sigma_K}+{\kappa}\right)\frac{d}{dt}\mathcal{N}_K
+\left(\frac{\ddot{\sigma}_K}{\sigma_K}
-\frac{2\dot{\sigma}_K^2}{\sigma_K^2}-2{\kappa}\frac{\dot{\sigma}_K}{\sigma_K}\right)\mathcal{N}_K.
\end{align}
\end{small}
Inspired by S. Li and X.-D. Li \cite{LiLi2}(see also \cite{LiLi3} for a survey),
we define the Perelman type $\mathcal{W}$-entropy by
\begin{small}
\begin{align}\label{KWentropy}
\mathcal{W}_K(t):=&\frac1{\dot{\alpha}_K(t)}\frac d{dt}(\alpha_K(t)\mathcal{N}_K(t))
=\mathcal{N}_K+\beta_K(t)\frac{d}{dt}\mathcal{N}_K\notag\\
=&-\sigma_K\int_M\Big[\beta_K(\gamma-1)\Delta v+\big(1+(\log\sigma_K)'\beta_K\big)\Big]vu\,dV,\notag\\
=&\sigma_K\beta_K\int_M\left[\gamma\frac{|\nabla v|^2}{v}-\left(\frac{1}{\beta_K}+\frac{\dot{\sigma}_K}{\sigma_K}\right)\right]vu\,dV,
\end{align}
\end{small}
where $\beta_K(t)=\frac{\alpha_K}{\dot{\alpha}_K}$, then
$$
\frac{d}{dt}\mathcal{W}_K(t)=\beta_K\left(\frac{d^2}{dt^2}\mathcal{N}_K
+\frac{1+\dot{\beta}_K}{\beta_K}\frac{d}{dt}\mathcal{N}_K\right).
$$
Combining \eqref{Kpment1} and \eqref{Kpment3},
we have
\begin{small}
\begin{align}\label{wkpment1}
\frac{d}{dt}\mathcal{W}_K(t)\le&
-2(\gamma-1)\sigma_K\beta_K\int_M \left[|\nabla\nabla v|^2+({\rm Ric}+Kg)(\nabla v,\nabla v )+(\gamma-1)(\Delta v)^2\right]vu\,dV\notag\\
&+2\beta_K\left(\frac{\dot{\sigma}_K}{\sigma_K}+\frac{1+\dot{\beta}_K}{2\beta_K}
+{\kappa}\right)\frac{d}{dt}\mathcal{N}_K
+\beta_K\left(\frac{\ddot{\sigma}_K}{\sigma_K}
-\frac{2\dot{\sigma}_K^2}{\sigma_K^2}
-2{\kappa}\frac{\dot{\sigma}_K}{\sigma_K}\right)\mathcal{N}_K.
\end{align}
\end{small}
On the other hand,
\begin{align}\label{wkpment2}
(\gamma-1)\left|\nabla_i\nabla_jv+\frac{\eta_K(t)}{n(\gamma-1)}g_{ij}\right|^2
=&(\gamma-1)|\nabla\nabla v|^2+\frac{2\eta_K}{n}\Delta v+\frac{\eta_K^2}{n(\gamma-1)}.
\end{align}
Putting \eqref{wkpment2} into \eqref{wkpment1}, we get
\begin{small}
\begin{align}\label{wkpment3}
&\frac{d}{dt}\mathcal{W}_K(t)\notag\\
\le&
-\sigma_K\beta_K\int_M 2 (\gamma-1)\left[\left|\nabla_i\nabla_jv+\frac{\eta_K}{n(\gamma-1)}g_{ij}\right|^2+({\rm Ric}+Kg)(\nabla v,\nabla v )\right]vu\,dV\notag\\
&-\sigma_K\beta_K\int_M\left[2(\gamma-1)^2(\Delta v)^2+2(\gamma-1)\left((\log\sigma_K)'+\frac{1+\dot{\beta}_K}{2\beta_K}+{\kappa}
-\frac{2\eta_K}{n(\gamma-1)}\right)\Delta v\right]vu\,dV\notag\\
&-\sigma_K\beta_K\int_M\left[(\log\sigma_K)''+ ((\log\sigma_K)')^2+\frac{1+\dot{\beta}_K}{\beta_K}(\log\sigma_K)'-\frac{2\eta_K^2}{n(\gamma-1)}\right]vu\,dV.
\end{align}
\end{small}
Set $\lambda=(\log\sigma_K)'$ and choose a proper function $\eta_{K}(t)$ such that
\begin{equation}\label{etabetak}
\left\{
  \begin{array}{l}
    2\eta_K=\lambda+\frac{1+\dot{\beta}_K}{2\beta_K}+{\kappa}
-\frac{2}{n(\gamma-1)}\eta_K \\
   2\eta_K^2=\lambda'+ \lambda^2+\frac{1+\dot{\beta}_K}{\beta_K}\lambda-\frac{2}{n(\gamma-1)}\eta_K^2,
  \end{array}
\right.
\end{equation}
which is equivalent to
\begin{align}\label{etak}
0=&\eta^2_K-2\lambda\eta_K+\frac{a}{a+1}
\left(\lambda^2-\lambda'+2{\kappa}\lambda\right)\notag\\
=&(\eta_K-\lambda)^2-\frac{1}{a+1}\left(\lambda^2
+a\left(\lambda'-2{\kappa}\lambda\right)\right),
\end{align}
where $a=\frac{n(\gamma-1)}{n(\gamma-1)+2}$. Solving the equation \eqref{etak}, we get a special solution
\begin{equation}\label{lambdaetak}
\lambda=\eta_K=\frac{2a{\kappa}}{1-e^{-2{\kappa}t}}.
\end{equation}
Putting \eqref{lambdaetak} back to system \eqref{etabetak}, we have
$$
\frac{1+\dot{\beta}_K}{\beta_K}=2\kappa\coth (kt),\quad \beta_K=\frac{\sinh(2{\kappa}t)}{2{\kappa}}
$$
and
$$
\alpha_K={\kappa}\tanh({\kappa}t),\quad
\sigma_K=\left(e^{{\kappa}t}\frac{\sinh({\kappa}t)}{\kappa}\right)^{a}=\left(\frac{e^{2{\kappa}t}-1}{2\kappa}\right)^{a}.
$$

Therefore, from \eqref{wkpment3}, we obtain the entropy monotonicity formula,
\begin{small}\begin{align}\label{wkpment4}
\frac{d}{dt}\mathcal{W}_K(t)\le&
-\sigma_K\beta_K\int_M 2 (\gamma-1)\left[\left| \nabla_i\nabla_jv+\frac{\eta_K}{n(\gamma-1)}g_{ij}\right|^2+({\rm Ric}+Kg)(\nabla v,\nabla v )\right]vu\,dV\notag\\
&-\sigma_K\beta_K\int_M2\left[(\gamma-1)\Delta v+\eta_K\right]^2vu\,dV.
\end{align}
\end{small}
Thus, when ${\rm Ric}\ge-Kg$, $K\ge0$, $\sigma_K>0,\eta_K>0$, Perelman type entropy \eqref{KWentropy} is monotone decreasing  along the porous medium equation \eqref{PME}.
\end{proof}

In \cite{LiLi,LiLi2}, when $n\le m\in \mathbb{N}$, Li-Li~gave another proof for entropy formula of Witten Laplacian and explained its geometric meaning by warped product, so we can also obtain the entropy formula  \eqref{WKPMEent} for the weighted porous medium equation by the analogous method. In fact, once one can obtain the entropy formula on Riemannian manifold, the weighted version is a direct result by this warped product approach when $m>n$ is a positive integer.

\begin{proof}[\bf{Proof of Theorem \ref{WKPMEentropy}}]
Let $\overline{M}=M\times N$ be a warped product manifold equipped the metric
$$
g_{\overline{M}}=g_M+e^{-\frac{2f}q}g_N,\quad q={m-n},
$$
where $m,n,q$ are the dimensions of $\overline{M}, M, N$ respectively,
then the volume measure satisfies
\begin{equation}\label{volmeasure}
\,\,dV_{\overline{M}}=e^{-f}\,\,dV_M\otimes \,\,dV_N=\,d\mu\otimes \,\,dV_N.
\end{equation}
where $(N^q,g_N)$ is a compact Riemannian manifold. Let $\pi:M\times N\to M$ be a nature projection map, $\overline{X}$ and $X$ are vector fields on $\overline{M}$ and $M$ respectively, then \cite{Besse}
$$
{\rm Ric}_{\overline{M}}(\overline{X},\overline{X})=\pi^*\left({\rm Ric}_{M}(X,X)-qe^{\frac fq}\nabla\nabla e^{-\frac fq}(X,X)\right),
$$
that is
\begin{equation}\label{mBER}
{\rm Ric}_f^m(X,X)=\pi_*\left({\rm Ric}_{\overline{M}}(\overline{X},\overline{X})\right).
\end{equation}
Assume~$V_N(N)=1$, $\overline{\nabla}$~denotes Levi-Civita connection on $(\overline{M},g_{\overline{M}})$, for any ~$v\in C^2(M)$, by S.Li and X.-D.Li \cite{LiLi}, we know
\begin{equation}\label{warped}
\overline{\nabla}_i\overline{\nabla}_jv=\nabla_i\nabla_jv,\quad
\overline{\nabla}_{\alpha}\overline{\nabla}_{\beta}v=-\frac 1q g_{\alpha\beta}g^{kl}\partial_kv\partial_lf,\quad
\overline{\nabla}_i\overline{\nabla}_{\alpha}v=0.
\end{equation}
where $i,j=1,2,\cdots,n$, and $\alpha,\beta=n+1,n+2,\cdots,m$, moreover,
\begin{equation}\label{warpedLaplacian}
\Delta_{\overline{M}}=\Delta_f+e^{-\frac{2f}{q}}\Delta_N.
\end{equation}
Thus, if~$u:M\to[0,\infty)$~ is a positive solution to ~\eqref{PME},
then~$u$~ satisfies the following equation
$$
\frac{\partial u}{\partial t}=\bar{\Delta}(u^{\gamma}).
$$
Moreover, in term of \eqref{volmeasure}, the weighted entropy functional~\eqref{WKentropy}~is indeed the functional on ~$(\overline{M},g_{\overline{M}})$
$$
\overline{\mathcal{W}}_K(v,t)=\sigma_K\beta_K\int_{\overline{M}}\left[\gamma\frac{|\nabla v|^2}{v}-\left(\frac{1}{\beta_K}+\frac{\dot{\sigma}_K}{\sigma_K}\right)\right]vu\,dV_{\overline{M}},
$$
Applying the entropy formula on $(\overline{M},g_{\overline{M}})$ in Theorem \ref{KPMEentropy}, we have
\begin{small}
\begin{align}\label{KPMEnt}
&\frac{d}{dt}\overline{\mathcal{W}}_K(t)\notag\\
\le&
-2(\gamma-1)\bar{\sigma}_K\bar{\beta}_K\int_{\overline{M}}\left(\left| \bar{\nabla}_i\bar{\nabla}_jv+\frac{\bar{\eta}_K}{m(\gamma-1)}\bar{g}_{ij}\right|_{\bar{g}}^2
+(\overline{{\rm Ric}}+K\bar{g})(\bar{\nabla} v,\bar{\nabla} v )\right)vu\,dV_M dV_N\notag\\
&-2\bar{\sigma}_K\bar{\beta}_K\int_{\overline{M}}\left((\gamma-1)\bar{\Delta} v+\eta_K\right)^2vu\,dV_M dV_N.
\end{align}
\end{small}
where
$$
\bar{\eta}_K=\frac{2\bar{a}{\kappa}}{1-e^{-2{\kappa}t}},\quad\bar{\beta}_K=\frac{\sinh(2{\kappa}t)}{2{\kappa}},\quad
\bar{\sigma}_K=\left(\frac{e^{2{\kappa}t}-1}{2\kappa}\right)^{{\bar{a}}},\quad \bar{a}=\frac{m(\gamma-1)}{m(\gamma-1)+2}.
$$
An analogous calculation in \cite{LiLi} gives,
\begin{small}
\begin{align}\label{WHessian}
\Big|\overline{\nabla}_i\overline{\nabla}_jv+\frac{\bar{\eta}_K}{m(\gamma-1)}\bar{g}_{ij}\Big|^2
=&\Big|\nabla_i\nabla_jv+\frac{\bar{\eta}_K}{m(\gamma-1)}g_{ij}\Big|^2+\Big|\overline{\nabla}_{\alpha}\overline{\nabla}_{\beta}v+\frac{\bar{\eta}_K}{m(\gamma-1)}g_{\alpha\beta}\Big|^2\\
=&\Big|\nabla_i\nabla_jv+\frac{\bar{\eta}_K}{m(\gamma-1)}g_{ij}\Big|^2+\frac{1}{m-n}\Big|\nabla v\cdot\nabla f-\frac{\bar{\eta}_K(m-n)}{m(\gamma-1)}\Big|^2
\end{align}
\end{small}
Combining ~\eqref{mBER}~with~\eqref{warpedLaplacian}~, one has
\begin{equation}\label{BERic}
{\rm \overline{Ric}}(\overline{\nabla} v,\overline{\nabla} v)={\rm Ric}^m_f(\nabla v,\nabla v),
\end{equation}
and
\begin{equation}\label{WLaplacian}
\bar{\Delta}v=\Delta_{f}v.
\end{equation}
Put \eqref{WHessian}, \eqref{BERic} and \eqref{WLaplacian} into \eqref{KPMEnt}, we can get the  desired result~\eqref{WKPMEent}.
\end{proof}

\section{Differential Harnack estimates and applications}
In this section, we prove a Qian type differential Harnack estimates \cite{QianB} for the porous medium equation on the compact Riemannian manifolds with lower bound of Ricci curvature, which generalize the Lu-Ni-Vazquez-Villani's estimate \cite{LNVV} and Huang-Huang-Li's estimate \cite{HHL}, Harnack inequalities and Laplacian estimate are derived as applications.
\begin{proposition} Let  $\sigma(t)$ be a functions of $t$ and satisfy the assumption (A1)(A2) and $\alpha(t), \varphi(t)$ are defined in \eqref{PMEalphavarphi}.  Suppose~$u$~be a smooth solution to \eqref{PME} and $v=\frac{\gamma }{\gamma-1}u^{\gamma-1}$ the pressure funciton. Define
$$
F_{\alpha}:=\alpha(t)\frac{v_t}{v}-\frac{|\nabla v|^2}{v}+\varphi(t).
$$
Set ${\kappa}=K\sup\limits_{M\times(0,T]}(u^{\gamma-1})$, for any $\gamma>1$, we have
\begin{small}
\begin{align}\label{GWDBochner1}
\square F_{\alpha}\ge&2\gamma \left\langle \nabla v ,\nabla F_{\alpha}\right\rangle+\frac1a\left((\gamma-1)\Delta v+\frac{a\sigma'}{2\sigma}+{a\gamma\kappa}\right)^2-\frac{\sigma'}{\sigma}F_{\alpha}
+(\alpha-1)\Big(\frac{v_t}v\Big)^2,
\end{align}
\begin{align}\label{GWDBochner2}
\square(\sigma F_{\alpha})\ge&2\gamma\sigma \left\langle \nabla v ,\nabla F_{\alpha}\right\rangle+\frac{\sigma}a\left((\gamma-1)\Delta v+\frac{a\sigma'}{2\sigma}+{a\gamma\kappa}\right)^2
+(\alpha-1)\sigma\Big(\frac{v_t}v\Big)^2.
\end{align}
\end{small}
\end{proposition}

\begin{proof} By using of \eqref{pmeBochner4}, we have
 \begin{align}\label{WDBochner4}
\square F_{\alpha}=&2\gamma\left\langle \nabla v ,\nabla F_{\alpha}\right\rangle+2(\gamma-1)\Big(|\nabla\nabla v|^2+{\rm Ric}(\nabla v,\nabla v)\Big)\notag\\
&+\left(((\gamma-1)\Delta v)^2+(\alpha-1)\Big(\frac{v_t}v\Big)^2\right)
+\alpha'\Big(\frac{v_t}{v}\Big)+\varphi'.
\end{align}
 Applying Cauchy-Schwartz inequality and \eqref{WDBochner4}, when ${\rm Ric}\ge-Kg$, we get
\begin{small}
\begin{align}\label{pmegradientest}
&\square F_{\alpha}-2\gamma \left\langle \nabla v ,\nabla F_{\alpha}\right\rangle\notag\\
\ge&\frac{1}{a}((\gamma-1)\Delta v)^2-2(\gamma-1)K|\nabla v|^2+(\alpha-1)\Big(\frac{v_t}v\Big)^2
+\alpha'\Big(\frac{v_t}{v}\Big)+\varphi'\notag\\
=&\frac{(\gamma-1)^2}{a}(\Delta v+\eta)^2-\frac{2 (\gamma-1)^2}{a}\eta\Delta v-\frac{(\gamma-1)^2}{a}\eta^2-2(\gamma-1)K|\nabla v|^2\notag\\
&+(\alpha-1)\Big(\frac{v_t}v\Big)^2
+\alpha'\Big(\frac{v_t}{v}\Big)+\varphi'\notag\\
\ge&\frac{(\gamma-1)^2}{a}(\Delta v+\eta)^2+\Big(\alpha'-\frac{2 (\gamma-1)}{a}\eta\Big)\frac{v_t}{v}-\Big(2\gamma{\kappa}
-\frac{2 (\gamma-1)}{a}\eta\Big)\frac{|\nabla v|^2}{v}\notag\\
&-\frac{(\gamma-1)^2}{a}\eta^2+(\alpha-1)\Big(\frac{v_t}v\Big)^2
+\varphi'\notag\\
=&\frac{(\gamma-1)^2}{a}(\Delta v+\eta)^2+\Big(2\gamma{\kappa}-\frac{2 (\gamma-1)}{a}\eta\Big)\left(\frac{\alpha'-\frac{2(\gamma-1)}{a}\eta}{2\gamma{\kappa}-\frac{2(\gamma-1)}{a}\eta}\frac{v_t}v- \frac{|\nabla v|^2}{v}+\varphi\right)\notag\\
&+(\alpha-1)\Big(\frac{v_t}v\Big)^2+\varphi'-\Big(2\gamma{\kappa}-\frac{2 (\gamma-1)}{a}\eta\Big)\varphi-\frac{(\gamma-1)^2}{a}\eta^2.
\end{align}
\end{small}

Choosing the proper functions $\sigma(t)$ and $\eta(t)$ such that $\alpha(t)$ and $\varphi(t)$ satisfy the following system
\begin{equation}\label{sigeta}
\left\{
  \begin{array}{rl}
\frac{\sigma'}{\sigma}&=\frac{2(\gamma-1)}{a}\eta-2\gamma{\kappa} \\
\alpha&=\frac{\alpha'-\frac{2(\gamma-1)}{a}\eta}{2\gamma{\kappa} -\frac{2 (\gamma-1)}{a}\eta} \\
\eta^2&=\frac{a}{(\gamma-1)^2}\left(\varphi'-\Big(2\gamma{\kappa}-\frac{2 (\gamma-1)}{a}\eta\Big)\varphi\right).
  \end{array}
\right.
\end{equation}
Plugging \eqref{sigeta} into \eqref{pmegradientest}, we get
$$
\square F_{\alpha}\ge2\gamma\left\langle \nabla v ,\nabla F_{\alpha}\right\rangle
+\frac{(\gamma-1)^2}{a}(\Delta v+\eta)^2-\frac{\sigma'}{\sigma}F_{\alpha}
+(\alpha-1)\Big(\frac{v_t}v\Big)^2,
$$
that is \eqref{GWDBochner1}.
Inequality \eqref{GWDBochner2} is a direct result of \eqref{GWDBochner1} and
$$
\square G=\square (\sigma F_{\alpha})=\sigma\square F_{\alpha}+\sigma' F_{\alpha}.
$$
In fact, the first equation in \eqref{sigeta} is equivalent to
$
\frac{2(\gamma-1)}{a}\eta(t)=\frac{\sigma'}{\sigma}+2\gamma{\kappa}.
$
Put this into  the last two equations in \eqref{sigeta}, we have
$$
\left\{
  \begin{array}{rl}
(\sigma\alpha)'=&\sigma'+2\gamma{\kappa}\sigma\\
(\sigma\varphi)'=&\frac{a\sigma}{4}\Big(\frac{\sigma'}{\sigma}+2\gamma{\kappa}\Big)^2,
  \end{array}
\right.
$$
Integral above identities on $[0,t]$, we can obtain the explicit expressions of $\alpha(t)$ and $\varphi(t)$ in \eqref{PMEalphavarphi}.
\end{proof}

\begin{proof}[\bf{Proof of Theorem \ref{pmeGK}}]
Since $\gamma>1$, $a>0$ and $M$ is closed,  the standard parabolic maximum principle in \eqref{GWDBochner2} implies $F_{\alpha}\ge0$, that is \eqref{PMELYH1} in Theorem \ref{pmeGK}.
\end{proof}

\begin{proof}[\bf{Proof of Corollary \ref{Harnack}}]
Let $\varsigma(t)$ be a constant speed geodesic with $\varsigma(t_1)=x_1$ and $\varsigma(t_2)=x_2$ such that $|\dot{\varsigma}(t)|=\frac{d(x_2,x_1)}{t_2-t_1}$. Using differential Harnack estimate \eqref{PMELYH1} and Young inequality, we have
\begin{small}
\begin{align*}
v(x_2,t_2)-v(x_1,t_1)=&\int^{t_2}_{t_1}v_t+\langle\nabla v,\dot{\varsigma}(t)\rangle dt\\
\ge&\int^{t_2}_{t_1}\left(\frac1{\alpha(t)}|\nabla v|^2-\frac{\varphi(t)}{\alpha(t)}v-\frac1{\alpha(t)}|\nabla v|^2-\frac{1}{4}\alpha(t)|\dot{\varsigma}(t)|^{2}\right)dt\\
\ge&-v_{max}\int^{t_2}_{t_1}\frac{\varphi(t)}{\alpha(t)}dt
-\frac{1}{4}\frac{d(x_2,x_1)^{2}}{(t_2-t_1)^{2}}\int^{t_2}_{t_1}\alpha(t)dt
\end{align*}
\end{small}
and
\begin{small}
\begin{align*}
\log\frac{v(x_2,t_2)}{v(x_1,t_1)}
=&\int^{t_2}_{t_1}\left(\frac{d}{dt}\log v(x,t)+\nabla\log v\cdot\dot{\varsigma}(t)\right)dt\\
\ge&\int^{t_2}_{t_1}\left(\frac{1}{\alpha(t)}\Big(|\nabla v|^2-\varphi(t)\Big)
-\frac1{\alpha(t)}|\nabla v|^2-\frac{1}{4}\frac{|\dot{\varsigma}(t)|^{2}}{v_{max}}\alpha(t)\right)dt\\
\ge&-\int^{t_2}_{t_1}\frac{\varphi(t)}{\alpha(t)}dt
-\frac{1}{4}\frac1{v_{max}}\frac{d(x_2,x_1)^{2}}{(t_2-t_1)^{2}}\int^{t_2}_{t_1}\alpha(t)dt.
\end{align*}
\end{small}
This finishes  the proof of Corollary \ref{Harnack}.
\end{proof}

\begin{proof}[\bf{Proof of Corollary \ref{pfLaEst}}]
Since $\alpha(t)>1$, direct calculation implies that
\begin{small}
\begin{align*}
\Delta(v^{\beta})=&\beta v^{\beta-1}\left(\Delta v+(\beta-1)\frac{|\nabla v|^2}{v}\right)\\
=&\frac{1}{\alpha (\gamma-1) }\beta v^{\beta-1}\left(\alpha (\gamma-1)\Delta v+\alpha(\gamma-1)(\beta-1)\frac{|\nabla v|^2}{v}\right)\\
=&\frac{1}{\alpha (\gamma-1) }\beta v^{\beta-1}\left(\alpha \frac{v_t}{v}-\frac{|\nabla v|^2}{v}\right).
\end{align*}
\end{small}
The estimate \eqref{pfLaE} follows from \eqref{PMELYH1}.
\end{proof}

\section*{Acknowledgments} The author would like to thank Professor Xiang-Dong Li and Dr. Songzi Li for their interest and illuminating discussions. The author is also thankful to the anonymous reviewers for their constructive comments and suggestions on the earlier version for this paper.


\end{document}